\documentclass[12pt]{amsart}

\usepackage{amsmath}
\usepackage{amsthm}
\usepackage{amssymb}
\usepackage{amscd}
\usepackage{amsfonts}
\usepackage{epsfig}

\newtheorem{thm}{Theorem}[section]
\newtheorem{lem}[thm]{Lemma}

\newtheorem{cor}[thm]{Corollary}

\theoremstyle{remark}
\newtheorem{rem}[thm]{Remark}

\def\gq{\mathfrak q}
\def\gp{\mathfrak p}
\def\gP{\mathfrak P}
\def\gm{\mathfrak m}
\def\gx{\mathfrak x}
\def\gy{\mathfrak y}
\def\A{\mathcal A}
\def\O{\mathcal O}
\def\G{\mathcal G}
\def\D{\mathcal D}
\def\M{\mathcal M}
\def\N{\mathcal N}
\def\Y{\mathcal Y}
\def\MT{\mathcal{MT}}
\def\Abar{\overline{A}}

\def\Nbar{\overline{N}}
\def\ebar{\overline{\eta}}

\def\zbar{\overline{z}}

\def\q{\mathbb Q}
\def\f{\mathbb F}
\def\ad{{\mathbb A}^\infty}
\def\ap{{\mathbb A}^{\infty,p}}
\def\c{\mathbb C}
\def\g{\mathbb G}

\def\p{\mathbb P}
\def\r{\mathbb R}

\def\z{\mathbb Z}

\def\zd{\mathaccent94\z}
\def\Sd{\mathaccent94\Sigma}
\def\Id{\mathaccent94I}

\def\Rd{\mathaccent94 R}

\def\End{\operatorname{End}}
\def\Aut{\operatorname{Aut}}
\def\GL{\operatorname{GL}}
\def\GSp{\operatorname{GSp}}
\def\U{\operatorname{U}}
\def\Ext{\operatorname{Ext}}
\def\Hom{\operatorname{Hom}}
\def\Mat{\operatorname{Mat}}
\def\Lie{\operatorname{Lie}}
\def\Gal{\operatorname{Gal}}
\def\Spec{\operatorname{Spec}}
\def\Spf{\operatorname{Spf}}
\def\Df{\operatorname{Defo}}
\def\tr{\operatorname{tr}}

\def\dm{\operatorname{dim}}
\def\ho{\operatorname{ht}}

\title{Embeddings of the complex ball into Siegel space}
\author{Oliver B\"ultel}

\address{Mathematisches Institut der Universit\"at Heidelberg, 
Im Neuenheimer Feld 288, 69120 Heidelberg, Germany,
bueltel@mathi.uni-heidelberg.de, fax:+49 6221 548312}


\begin{document}
\maketitle

\begin{abstract}
We study properties of a certain map from the unitary group $U(1,n-1)$ to
the group $\U(\binom{n-1}{k-1},\binom{n-1}{k})$. We explain how it gives rise 
to a map between canonical models of Shimura varieties and we prove that it 
extends to the ordinary locus of the integral model. Finally, we extend results
of Satake on the endomorphism ring of a generic image point to positive 
characteristics.
\end{abstract}

\keywords

\section{Introduction}

In ~\cite{satake} Satake classifies embeddings of a symmetric Hermitian space 
$X$ into Siegel space, let us describe his results in a special case: The 
bounded realization of $I_{p,n-p}$ is a domain in complex affine 
$p(n-p)$-space which may conveniently be described as the set of complex
valued $p\times(n-p)$-matrices
$$
B_{p,n-p}=\{A\in\Mat(p,n-p)|\mbox{ }||A||_\infty<1\}
$$
where $||\cdot||_\infty$ is the operatornorm (relative to the standard Hilbert 
space norms $||\cdot||_2$ on $\c^p$ and $\c^{n-p}$ respectively). If we write 
$S_n$ for the bounded realization of $III_n$, then there is a natural 
holomorphic embedding 
\begin{equation}
\label{einb}
B_{p,n-p}\hookrightarrow S_n.
\end{equation}
In the special case $p=1$ we recover the complex ball in $\c^{n-1}$. 
Surprisingly ~\eqref{einb} is not the sole symplectic embedding of 
$B_{1,n-1}$, as one has maps
\begin{equation}
\label{zweib}
B_{1,n-1}\hookrightarrow B_{\binom{n-1}{k-1},\binom{n-1}{k}}
\hookrightarrow S_{\binom{n}{k}}
\end{equation}
which are induced by sending $x\in\c^{n-1}$ to the 
$\binom{n-1}{k-1}\times\binom{n-1}{k}$-matrix $A$ with entries
$$
A_{I,J}=
\begin{cases}
(-1)^{\nu-1}x_{i_\nu}&I=J-\{i_\nu\}\\
0&\text{otherwise}
\end{cases}
$$
where $I=\{i_1,\dots,i_{k-1}\}$ and $J=\{i_1,\dots,i_k\}$ run 
through all subsets of $\{1,\dots,n-1\}$ with cardinalities $k-1$ and $k$. We 
want to study the effect of ~\eqref{zweib} to quotients of $B_{1,n-1}$ by 
congruence subgroups and their canonical models. We consider, in more modern 
language, a Shimura datum $(G,X)$, such that $G$ is a $\q$-form of $\GL(n,\c)$ 
and such that the conjugacy class $\mu$ of minuscule cocharacters obtained 
from $X$ is:
$$
\c^\times\ni z\mapsto\left(\begin{matrix} 
z&0&\dots&0\\
0&1&\dots&0\\
\vdots&\vdots&\ddots&\vdots\\
0&0&\dots&1
\end{matrix}\right)\in\GL(n,\c).
$$
If $\mu$ is as above then all of the $n-1$ fundamental dominant weights 
$\lambda$ satisfy Deligne's condition $<\mu,\lambda+c(\lambda)>=1$, $c$ is 
the Weil-opposition. Thus one expects by ~\cite[1.3]{deligne3} that (suitable 
$\q$ forms of) their highest weight representations provide maps from the 
Shimura datum $(G,X)$ to a symplectic Shimura datum 
$(\GSp(2\binom{n}{k},\q),S_{\binom{n}{k}}^\pm)$. Such a symplectic 
representation of $(G,X)$ gives rise to the map ~\eqref{zweib} of the 
underlying symmetric Hermitian space, which in our case is the complex ball. 
Now consider a neat level structure $K\subset G(\ad)$. By Deligne's theory we 
obtain a Shimura variety
\begin{equation}
\label{drei}
M(\c)=G(\q)\backslash(X\times G(\ad)/K)
\end{equation}
which has weakly canonical models over all fields containing the reflex. 
Moreover, the aforementioned symplectic maps provide us with certain abelian 
schemes $Y^{(k)}$ over $M$. (Here one has to allow to replace $(G,X)$ by a 
certain covering Shimura datum denoted $(G_1,X_1)\rightarrow(G,X)$ in 
~\cite[Proposition 2.3.10]{deligne3}, see body of text for details.) If one 
assumes that $G$ and $K$ are sufficiently well behaved at $p$, then the theory 
of $PEL$-moduli spaces provides us with a smooth model $\M$.\\
In this paper we study the extension properties of $Y^{(k)}$ with respect to 
the model $\M$. Our result is that $Y^{(k)}$ extends to an abelian scheme at 
least over the ordinary locus of $\M$. Our result holds for all primes 
including $2$.\\ 
In a future publication we will show that one can remove the ordinarity 
hypothesis at the cost of having to restrict to primes bigger than $k+1$ along 
with somewhat stronger conditions on $G$ and $K$. More general results (under 
similar restrictions on the prime) have already been shown in \cite{vasiu}.\\
The work is organised as follows: In section ~\ref{Hodge} we describe the 
Shimura datum $(G_1,X_1)$ and the abelian schemes $Y^{(k)}$ using the 
language of Hodge structures, we do not attempt to work with all possible 
$\q$-forms of $G$ and of $G_1$ but rather confine ourself to what we think 
is the most natural one. This allows us to give a very down to earth (and we 
hope enlightening) treatment of the covering Shimura datum $(G_1,X_1)$ 
(denoted $(G^{(0\times1)},X^{(0\times1)})$ in body of text), and of its 
symplectic representations (which we denote by $V^{(k)}$). In section 
~\ref{deform} we recall canonical coordinates and canonical lifts, to a 
large extend with sketchy proof. In section ~\ref{good} we show that under 
the usual assumptions on $K$ the variety $M$ has a smooth model over the
ring of integers in the local field $E_\gp$, where $E$ is a certain
finite extension of the reflex field and where $\gp$ is a prime over $p$. In 
section ~\ref{lego} we prove the extension theorem in two steps, we start 
with the construction of an extension of the Barsotti-Tate group of 
$Y^{(k)}$, and then we show that this determines a well-defined extension 
of $Y^{(k)}$ itself, essentially by Zariski's main theorem. Finally, we
show in section ~\ref{ende} how some of Satake's work on the endomorphism 
ring of $Y^{(k)}$ can be carried over to the points in the ordinary
locus of the special fibre of $\M$.\\

I thank Prof. Ivan Fesenko, Prof. Eberhard Freitag, Prof. Winfried Kohnen, 
Prof. Richard Pink, Prof. Richard Taylor, and Prof. Rainer Weissauer for 
interesting conversations on the topic and further thanks go to the referee.

\section{Exterior powers of Hodge structures}
\label{Hodge}

Let $V$ be a finitely generated torsion free abelian group. It is called a 
($\z$-) Hodge structure of weight $r$ if it is equipped with a decomposition
$$
V_\c\cong\bigoplus_{p+q=r}V^{p,q}
$$
satisfying $V^{q,p}=\overline{V^{p,q}}$. A direct sum
$$
V\cong\bigoplus_rV_r,
$$
of ($\z$-) Hodge structures $V_r$ of weight $r$, we want to call a pure 
($\z$-) Hodge structure. To give the decomposition over the reals is 
equivalent to give a homomorphism of $\r$-algebraic groups
$$
h:\c^\times\rightarrow\GL_\r(V_\r)\subset\GL_\c(V_\c)
$$
by making $z\in\c^\times$ act on $V^{p,q}$ by $z^{-p}\zbar^{-q}$. If the only 
non-zero direct summands amongst the $V^{p,q}$'s are $V^{-1,0}$ and 
$V^{0,-1}$ then the homomorphism above corresponds to a homomorphism between 
the $\r$-algebras $\c$ and $\End_\r(V_\r)$, which we continue to denote by 
$h$. In that case $V$ is called a Hodge structure of type $\{(-1,0),(0,-1)\}$,
and the $\c$-space $V^{-1,0}$ can set-theoretically be identified with 
$V_\r$ which acquires its $\c$-vector space structure by the 
corresponding algebra homomorphism $h:\c\rightarrow\End_\r(V_\r)$. Pure Hodge 
structures form an additive $\otimes$-category in an obvious way. The pure 
Hodge structures of type $\{(-1,0),(0,-1)\}$ are an additive full subcategory 
thereof.\\
Let $\O_L$ be the ring of integers in a totally imaginary quadratic extension 
$L$ of a totally real number field $L^+$. Write $*$ for the non-trivial 
involution of $L/L^+$. A Hodge structure $V$ is said to have a 
$\O_L$-operation if a homomorphism $\iota:\O_L\rightarrow\End(V)$ is given. 
If $\iota$ is fixed we obtain a refinement of the Hodge decomposition:
$$
V_\sigma=V\otimes_{L,\sigma}\c\cong\bigoplus_{p,q}V_\sigma^{p,q},
$$
where we denote for any embedding $\sigma:L\rightarrow\c$ by $V_\sigma^{p,q}$ 
the $\sigma$-eigenspace of $V^{p,q}$. Note that we have 
$V^{q,p}_\sigma\cong\overline{V_{\sigma\circ*}^{p,q}}$. The Hodge structures 
with $\O_L$-operation form a $\otimes$-category as follows: If $V$ and $W$ are 
given, we can form the finitely generated torsion free $\O_L$-module 
$V\otimes_{\O_L}W$. We make it into a Hodge structure by declaring the
subspace of $V_\sigma\otimes_\c W_\sigma$ of weight $(p,q)$ to be:
$$
\bigoplus_{p_1,q_1}
V_\sigma^{p_1,q_1}\otimes_\c W_\sigma^{p-p_1,q-q_1}.
$$
In a similar way we define $\bigwedge_{\O_L}^kV$. We want to describe a 
particularly interesting example of such tensor constructions, to this end 
fix a natural number $n$. Fix two different $CM$-traces 
$\Phi^{(0)}$, $\Phi^{(n)}$: $L\rightarrow\c$ ($=\q$-linear maps which arise as 
$\Phi^{(k)}(x)=\sum_{\sigma\in|\Phi^{(k)}|}\sigma(x)$ for all $x\in L$ and 
$k\in\{0,n\}$, where $|\Phi^{(0)}|,|\Phi^{(n)}|\subset\Hom_\q(L,\c)$). Let us 
define:
\begin{equation}
\label{vfuenf}
\Phi^{(1)}(x)=(n-1)\Phi^{(0)}(x)+\Phi^{(n)}(x)
\end{equation}
Our input is a choice of:
\begin{itemize}
\item
two Hodge structures of type $\{(-1,0),(0,-1)\}$ with $\O_L$-operation, i.e. 
$\r$-linear homomorphisms
\begin{equation}
\label{heins}
h^{(k)}:\c\rightarrow\End_L(V_\r^{(k)}),
\end{equation}
for $k\in\{0,1\}$, where the $V^{(k)}$ are finitely generated, torsion free 
$\O_L$-modules. We require that:
\begin{equation}
\label{veins}
\tr(x|_{{V^{(k)}}^{-1,0}})=\Phi^{(k)}(x),
\end{equation}
hold for all $x\in L$, and $k\in\{0,1\}$.
\item
two $*$-skew-Hermitian $\q$-valued forms $\psi^{(1)}$, and $\psi^{(0)}$ on 
the $L$-spaces $V_\q^{(1)}$, and $V_\q^{(0)}$, such that firstly the $h^{(k)}$ 
become $*$-involution preserving homomorphisms, if the right hand side of 
~\eqref{heins} is endowed with the Rosati involution, and such that secondly 
the forms $\psi^{(k)}(x,h^{(k)}(i)y)$ on $V_\r^{(k)}$ are positive definite.
\end{itemize}
Observe that ~\eqref{veins} implies that the $\O_L$-modules $V^{(1)}$ and 
$V^{(0)}$ have the ranks $n$ and $1$. Observe also that we may write 
the forms $\psi^{(1)}$ and $\psi^{(0)}$ in a unique way as traces 
$\tr_{L/\q}\Psi^{(1)}$, and $\tr_{L/\q}\Psi^{(0)}$, where $\Psi^{(1)}$, and 
$\Psi^{(0)}$ are sesquilinear forms that are $L$-valued on $V_\q^{(1)}$, and 
on $V_\q^{(0)}$. Based on the 2nd and 3rd row of the table on page 188 of 
~\cite[Chapter IV,Paragraph 5]{satake} we want to consider the following 
output:
 
\begin{enumerate}
\item
$\O_L$-modules $V^{(k)}$ defined by: 
$$
{V^{(0)}}^{\otimes_{\O_L}1-k}\otimes_{\O_L}\bigwedge_{\O_L}^kV^{(1)},
$$
for every $k\in\{0,\dots,n\}$,
\item
$\q$-valued $(L,*)$-skew-Hermitian forms 
$\psi^{(k)}=\tr_{L/\q}\Psi^{(k)}$ on $V_\q^{(k)}$, here is:
$$
\Psi^{(k)}(x_0^{1-k}x_1\wedge\dots\wedge x_k,y_0^{1-k}y_1\wedge\dots\wedge y_k)
$$
$$
=(\Psi^{(0)}(x_0,y_0))^{1-k}\det(\Psi^{(1)}(x_i,y_j)_{i,j}),
$$
for $x_1,\dots,x_k,y_1,\dots,y_k\in V_\q^{(1)}$, and $x_0,y_0\in V_\q^{(0)}$,
\item
$*$-homomorphisms $h^{(k)}:\c^\times\rightarrow\Aut(V_\r^{(k)})$ defined so 
that $z\in\c^\times$ acts as:
$$
x_0^{1-k}x_1\wedge\dots\wedge x_k\mapsto
(h^{(0)}(z)x_0)^{1-k}h^{(1)}(z)x_1\wedge\dots\wedge h^{(1)}(z)x_k.
$$
\end{enumerate}
We have to introduce reductive $\q$-groups as follows, on the category of
all $\q$-algebras $G^{(k)}$ represents the functor:
$$
C\mapsto\{(\gamma,\mu)\in\End_{L\otimes C}^\times(V_C^{(k)})\times C^\times|
\psi^{(k)}(\gamma x,\gamma y)=\mu\psi^{(k)}(x,y)\}
$$
Notice the natural similitude morphisms from $G^{(k)}$ to $\g_m$. We also 
introduce a group $G^{(0\times 1)}$ by the requirement that the diagram
$$\begin{CD}
G^{(0\times 1)}@>g^{(1)}>>G^{(1)}\\
@Vg^{(0)}VV@VVV\\
G^{(0)}@>>>\g_m
\end{CD}
$$
be cartesian. Finally we introduce $\q$-rational group homomorphisms
$g^{(k)}:G^{(0\times1)}\rightarrow G^{(k)}$ by sending, say 
$(\gamma^{(0)},\gamma^{(1)})\in G^{(0\times1)}(C)$ to the element 
$\gamma^{(k)}:V_C^{(k)}\rightarrow V_C^{(k)}$ defined by
\begin{equation}
\label{vsechs}
\gamma^{(k)}(x_0^{1-k}x_1\wedge\dots\wedge x_k)=
\gamma^{(0)}(x_0)^{1-k}\gamma^{(1)}(x_1)\wedge\dots\wedge\gamma^{(1)}(x_k).
\end{equation}
Let us also pick connected, smooth $\z$-models $G_\z^{(k)}$ 
($G_\z^{(0\times1)}$) of the groups $G^{(k)}$ ($G^{(0\times1)}$) by taking 
the schematic closure in the group $\z$-schemes $\GL(V^{(k)}/\z)$ 
($\GL(V^{(0)}\oplus V^{(1)}/\z)$). Here the $\O_L$-lattices 
$V^{(k)}\subset V_\q^{(k)}$ 
($V^{(0)}\oplus V^{(1)}\subset V_\q^{(0)}\oplus V_\q^{(1)}$), are as before.
The maps $g^{(k)}$ preserve these lattices and so give rise to maps 
$g_\z^{(k)}:G_\z^{(0\times1)}\rightarrow G_\z^{(k)}$. The $V^{(k)}$ are 
clearly Hodge structures with $\O_L$-operation, however much more is true:

\begin{lem}
\label{DATA}
Let $(V^{(1)},V^{(0)},\psi^{(1)},\psi^{(0)},h^{(1)},h^{(0)})$ be as before. 
Then $(V^{(k)},h^{(k)})$ is a Hodge structure of type $\{(-1,0),(0,-1)\}$ 
with $\O_L$-operation. Moreover, the quadruples
$$
(L,*,V_\q^{(k)},\psi^{(k)},h^{(k)})
$$
are $PEL$-data in the sense of ~\cite{kottwitz}, as is the quadruple
$$
(L\oplus L,*,V_\q^{(0)}\oplus V_\q^{(1)},\psi^{(0)}
\oplus\psi^{(1)},h^{(0)}\oplus h^{(1)}).
$$
The $G^{(k)}(\r)$-conjugacy class of $h^{(k)}$ is determined by the formula:
\begin{equation}
\label{vzwei}
\tr(x|_{{V^{(k)}}^{-1,0}})=\Phi^{(k)}(x)=
\binom{n-1}{k}\Phi^{(0)}(x)
+\binom{n-1}{k-1}\Phi^{(n)}(x).
\end{equation}
for any $x\in L$. Therefore, when writing $E^\delta$ for the reflex 
field of $(G^\delta,h^\delta)$, we have $E^{(k)}\subset E^{(0\times1)}$
for all $k$. Finally, the map $g^{(k)}:G^{(0\times1)}\rightarrow G^{(k)}$ 
takes $h^{(0\times1)}:=h^{(0)}\oplus h^{(1)}$ to $h^{(k)}$. 
\end{lem}
\begin{proof}
We first check that the Hodge structure $V^{(k)}$ has weights in the set 
$\{(-1,0),(0,-1)\}$. The assumption ~\eqref{veins} on $V^{(1)}$ implies that 
the ${V_\sigma^{(1)}}^{-1,0}$ have the following dimensions, where 
$\sigma:L\rightarrow\c$ runs through all embeddings:
$$
\dm_\c{V_\sigma^{(1)}}^{-1,0}=
\begin{cases}
n&\text{ if $\sigma\in|\Phi^{(0)}|\cap|\Phi^{(n)}|$}\\
n-1&\text{ if $\sigma\in|\Phi^{(0)}|-|\Phi^{(n)}|$}\\
1&\text{ if $\sigma\in|\Phi^{(n)}|-|\Phi^{(0)}|$}\\
0&\text{ if $\sigma\notin|\Phi^{(0)}|\cup|\Phi^{(n)}|$}
\end{cases},
$$
and analogous formulas for $\dm_\c{V_\sigma^{(1)}}^{0,-1}$ hold. Consequently 
we derive the following Hodge decomposition for the $\sigma$-eigenspace of 
$\bigwedge_{\O_L}^kV^{(1)}$:
$$
\bigwedge_\c^kV_\sigma^{(1)}=
\begin{cases}
\bigwedge_\c^k{V_\sigma^{(1)}}^{-1,0}
&\text{if $\sigma\in|\Phi^{(0)}|\cap|\Phi^{(n)}|$}\\
\bigwedge_\c^k{V_\sigma^{(1)}}^{-1,0}\oplus
\bigwedge_\c^{k-1}{V_\sigma^{(1)}}^{-1,0}\otimes_\c{V_\sigma^{(1)}}^{0,-1}
&\text{if $\sigma\in|\Phi^{(0)}|-|\Phi^{(n)}|$}\\
{V_\sigma^{(1)}}^{-1,0}\otimes_\c\bigwedge_\c^{k-1}{V_\sigma^{(1)}}^{0,-1}
\oplus\bigwedge_\c^k{V_\sigma^{(1)}}^{0,-1}
&\text{if $\sigma\in|\Phi^{(n)}|-|\Phi^{(0)}|$}\\
\bigwedge_\c^k{V_\sigma^{(1)}}^{0,-1}
&\text{if $\sigma\notin|\Phi^{(0)}|\cup|\Phi^{(n)}|$}
\end{cases}
$$
so that in each of the four cases the $\sigma$-eigenspace of 
$\bigwedge_{\O_L}^kV^{(1)}$ has the following Hodge weights (with
multiplicity):
$$
(p,q)=
\begin{cases}
{\binom{n}{k}\times(-k,0)}
&\text{if $\sigma\in|\Phi^{(0)}|\cap|\Phi^{(n)}|$}\\
{\binom{n-1}{k}\times(-k,0), 
\binom{n-1}{k-1}\times(1-k,-1)}
&\text{if $\sigma\in|\Phi^{(0)}|-|\Phi^{(n)}|$}\\
{\binom{n-1}{k-1}\times(-1,1-k), 
\binom{n-1}{k}\times(0,-k)}
&\text{if $\sigma\in|\Phi^{(n)}|-|\Phi^{(0)}|$}\\
{\binom{n}{k}\times(0,-k)}
&\text{if $\sigma\notin|\Phi^{(0)}|\cup|\Phi^{(n)}|$}
\end{cases}
$$
Comparing this with the Hodge weights of ${V_\sigma^{(0)}}^{\otimes_\c1-k}$, 
being
$$
(p,q)=
\begin{cases}
(k-1,0)&\text{if $\sigma\in|\Phi^{(0)}|$}\\
(0,k-1)&\text{if $\sigma\notin|\Phi^{(0)}|$}
\end{cases}
$$
shows that the sole weights of $V^{(k)}$ are $\{(-1,0),(0,-1)\}$ and shows the 
formula ~\eqref{vzwei} also.\\
We still have to verify that the pairings $\psi^{(k)}(x,h^{(k)}(i)y)$ are 
positive definite. Every embedding $\sigma:L\rightarrow\c$ is a 
$*$-homomorphism, therefore one obtains sesquilinear $\c$-valued pairings 
$\Psi_\sigma^{(k)}$ on $V_\sigma^{(k)}$ for every $k$. Consider the forms 
$E_\sigma^{(k)}=\Psi_\sigma^{(k)}(x,h^{(k)}(i)y)$. By assumption 
$\psi^{(1)}(x,h^{(1)}(i)y)$, and $\psi^{(0)}(x,h^{(0)}(i)y)$ are positive 
definite on $V_\r^{(1)}$, and $V_\r^{(0)}$, so these forms restrict to 
positive definite forms $E_\sigma^{(1)}$, and $E_\sigma^{(0)}$ on 
$V_\sigma^{(1)}$, and $V_\sigma^{(0)}$. By utilizing the formulas for 
$\Psi^{(k)}$ and $h^{(k)}$ we see the positive definiteness of 
$E_\sigma^{(k)}$, as
$$
E_\sigma^{(k)}(x_0^{1-k}x_1\wedge\dots\wedge x_k,
y_0^{1-k}y_1\wedge\dots\wedge y_k)
=E_\sigma^{(0)}(x_0,y_0)^{1-k}\det(E_\sigma^{(1)}(x_i,y_j)_{i,j}).
$$
Finally just check that $\psi^{(k)}(x,h^{(k)}(i)y)$ is the sum of various
$E_\sigma^{(k)}(x,y)$'s.
\end{proof}
We turn to Shimura varieties which correspond to our choice of
$V_\q^\delta,\psi^\delta,h^\delta$, for 
$\delta\in\{(0\times1),(0),\dots,(n)\}$. Let $X^\delta$ be the conjugacy 
class of $h^\delta$ in $G_\r^\delta$, (we set 
$h^{(0\times1)}=h^{(0)}\oplus h^{(1)}$), then one puts according to 
~\cite{deligne1}: 
$$
{_KM}_\c(G^\delta,X^\delta)=G^\delta(\q)\backslash
(X^\delta\times G^\delta(\ad)/K)
$$
for the Shimura varieties, of which the projective limit
$\lim\limits_{K\to1}{_KM}_\c(G^\delta,X^\delta)$ is equal to
$$
M_\c(G^\delta,X^\delta)=
G^\delta(\q)\backslash(X^\delta\times G^\delta(\ad)),
$$
by ~\cite[Corollaire 2.1.11]{deligne3}. One writes ${_KM}(G^\delta,X^\delta)$ 
and $M(G^\delta,X^\delta)$, for their weakly canonical models over a choice 
of any field containing the reflex. We will prefer to always work over the 
field $E^{(0\times1)}$, being the largest of the reflex fields $E^\delta$, 
thus our $M(G^{(k)},X^{(k)})$'s are strictly speaking base changes via 
$\times_{E^{(k)}}E^{(0\times1)}$ of the canonical models.\\
Let us briefly sketch the moduli interpretations due to Deligne 
~\cite[Scholie 4.11]{deligne1}. If $F/E^{(0\times1)}$ is an algebraically 
closed field then there is a 
$\Aut_{E^{(0\times1)}}(F)\times G^{(k)}(\ad)$-equivariant bijection between 
$F$-valued points of $M(G^{(k)},X^{(k)})$ and isogeny classes of quadruples 
$(Y^{(k)},\iota^{(k)},\lambda^{(k)},\eta^{(k)})$ such that:
\begin{itemize}
\item[(a)]
$Y^{(k)}$ is an abelian variety over $F$ up to isogeny.
\item[(b)]
$\iota^{(k)}:L\rightarrow\End^0(Y^{(k)})$ is a homomorphism such that
$$
\tr_{\Lie Y^{(k)}}(\iota^{(k)}(x))=\Phi^{(k)}(x)
$$ 
for all $x\in L$.
\item[(c)]
$\lambda^{(k)}$ is a homogeneous polarization of $Y^{(k)}$ of which
the Rosati involution restricts to the $CM$ involution on $L$.
\item[(d)]
a $L$-linear level structure 
$\eta^{(k)}:{V_\ad}^{(k)}\rightarrow H_1(Y^{(k)},\ad)$ which becomes a 
symplectic similitude if one imposes the Weil pairing on the right and the 
pairing $\psi^{(k)}$ on the left.
\item[(e)]
the skew-Hermitian $L$-module $H_1(Y^{(k)}\times_F\c,\q)$ is isomorphic to
$(V_\q^{(k)},\psi^{(k)})$
\end{itemize}
and similarly for $M(G^{(0\times1)},X^{(0\times1)})$. In fact a more thorough
treatment can be found in ~\cite{kottwitz}: The scheme 
$M(G^{(0\times1)},X^{(0\times1)})$ represents a functor taking a 
$E^{(0\times1)}$-scheme $S$ to the set of tuples 
$(Y^{(k)},\iota^{(k)},\lambda^{(k)},\eta^{(k)})$ with properties analogous to 
(a)-(e). From now on let us pick a level $l$ which is $\geq3$. The compact 
open groups 
\begin{equation}
\label{co}
K^\delta=\{\gamma\in G_\z^\delta(\zd)|\gamma\equiv1\pmod l\}
\end{equation}
are neat and satisfy $g^{(k)}(K^{(k)})\subset K^{(0\times1)}$. This is most 
useful as it sets up maps between the characteristic zero Shimura varieties
\begin{equation}
\label{velf}
g^{(k)}:{_{K^{(0\times1)}}M}(G^{(0\times1)},X^{(0\times1)}) 
\rightarrow{_{K^{(k)}}M}(G^{(k)},X^{(k)})
\end{equation}
that are induced from the group theoretic maps of lemma ~\ref{DATA}, 
according to ~\cite[Corollaire 5.4]{deligne1}. Note that these maps do not 
have a natural moduli interpretation. Nevertheless we do obtain a
homogeneously polarized abelian scheme $Y^{(k)}$ up to isogeny over 
${_{K^{(0\times1)}}M}(G^{(0\times1)},X^{(0\times1)})$, with $\O_L$-operation
$\iota^{(k)}$, and appropriate level structure $\eta^{(k)}$, by pulling back 
the universal family on ${_{K^{(k)}}M}(G^{(k)},X^{(k)})$ via $g^{(k)}$. Here, 
we wish to consider the particular representative of the isogeny class 
$Y^{(k)}$ which is determined by the constraint 
$\eta^{(k)}(V_{\zd}^{(k)})=H_1(Y_\xi^{(k)},\zd)$. Within the homogeneous class
of polarizations $\lambda^{(k)}$, coming from the data $\q\psi^{(k)}$, we wish 
to pick some effective representative and write $d_k$ for its degree. Note 
that the homogeneous class $\q\psi^{(k)}$ depends only on the homogeneous 
class $\q(\psi^{(0)}\oplus\psi^{(1)})$ of polarizations on 
$V^{(0)}\oplus V^{(1)}$, but the choice of an effective representative is 
arbitrary. Note also that $\psi^{(k)}$ may well be ineffective even if 
$\psi^{(0)}\oplus\psi^{(1)}$ is effective, see lemma ~\ref{prinz} below, 
however. Having made the above choice we obtain a further map:
\begin{equation}
\label{fuenf}
g^{(k)}:{_{K^{(0\times1)}}M}(G^{(0\times1)},X^{(0\times1)}) 
\rightarrow\A_{g_k,d_k,l}
\end{equation}
where $g_k$ is $[L^+:\q]\binom{n}{k}$, and $\A_{g_k,d_k,l}$ is the fine moduli 
space of polarized abelian $g_k$-folds of degree $d_k$ with level 
$l$-structure. We finish this section with two more results on $Y^{(k)}$: 
\begin{lem}
\label{vvier}
Let $V^\delta,h^\delta,\psi^\delta,K^\delta$, be as above. Let 
$F/E^{(0\times1)}$ be a field and let 
$$
\xi:\Spec F\rightarrow{_{K^{(0\times1)}}M}(G^{(0\times1)},X^{(0\times1)}) 
$$
be a point, let $Y_\xi^{(k)}/F$ be the abelian varieties which correspond to 
$g^{(k)}(\xi)$ via (a)-(e) above (and leveled by the constraint 
$\eta^{(k)}(V_{\zd}^{(k)})=H_1(Y_\xi^{(k)},\zd)$). Then:
$$
H_1(Y_\xi^{(0)}\times_FF^{ac},\z_\ell)^{\otimes_{\O_{L_\ell}}1-k}
\otimes_{\O_{L_\ell}}
\bigwedge_{\O_{L_\ell}}^kH_1(Y_\xi^{(1)}\times_FF^{ac},\z_\ell)
$$
is isomorphic to
$$
H_1(Y_\xi^{(k)}\times_FF^{ac},\z_\ell)
$$
as a $\O_{L_\ell}[\Gal(F^{ac}/F)]$-module.
\end{lem}
\begin{proof}
We can lift $\xi$ to a point $\xi':\Spec F^{ac}\rightarrow 
M(G^{(0\times1)},X^{(0\times1)})$, using the moduli interpretation of 
$M^{(0\times1)}$ (see section ~\ref{good} below) we can find the associated 
tuple 
$$
(Y_\xi^{(1)},Y_\xi^{(0)},\iota^{(1)},\iota^{(0)},
\lambda^{(0\times1)},\eta^{(1)},\eta^{(0)})
$$
where $\lambda^{(0\times1)}$ is a homogeneous polarization of 
$Y_\xi^{(0)}\times Y_\xi^{(1)}$ and $\eta^{(0)},\eta^{(1)}$ are 
$L$-linear similitudes:
$$
\eta^{(1)}:{V_{\ad}}^{(1)}\rightarrow 
H_1(Y_\xi^{(1)},\ad).
$$
and
$$
\eta^{(0)}:{V_{\ad}}^{(0)}\rightarrow 
H_1(Y_\xi^{(0)},\ad)
$$
of which the multipliers agree. Now let $\tau$ be an element of 
$\Gal(F^{ac}/F)$, as $\xi'$ and $\tau(\xi')$ have the same image in 
${_{K^{(0\times1)}}M}(G^{(0\times1)},X^{(0\times1)})$ one can find 
$\gamma=(\gamma^{(0)},\gamma^{(1)})\in K^{(0\times1)}$ with 
$\tau(\xi')=\xi'.\gamma$. Due to the $G^{(0\times1)}(\ad)$-equivariance of the
map $M(G^{(0\times1)},X^{(0\times1)})\rightarrow M(G^{(k)},X^{(k)})$ it 
follows that $\tau(g^{(k)}(\xi'))=g^{(k)}(\xi').g^{(k)}(\gamma)$. However,
$\xi'.\gamma$ is by definition equal to: 
$$
(Y_\xi^{(1)},Y_\xi^{(0)},\iota^{(1)},\iota^{(0)},\lambda^{(0\times1)},
\eta^{(1)}\circ\gamma^{(1)},\eta^{(0)}\circ\gamma^{(0)}),
$$
and analogously for $g^{(k)}(\xi').g^{(k)}(\gamma)$. Now use the formula 
~\eqref{vsechs} and one is done.
\end{proof}

\begin{cor}
If the assumptions are as above, then:
$$
H_1^{dR}(Y_\xi^{(0)}/F)^{\otimes_{F\otimes L}1-k}
\otimes_{F\otimes L}\bigwedge_{F\otimes L}^kH_1^{dR}(Y_\xi^{(1)}/F)
$$
is isomorphic to
$$
H_1^{dR}(Y_\xi^{(k)}/F)
$$
as $F\otimes L$-module with Hodge filtration and Gauss-Manin connection.
\end{cor}
\begin{proof}
Without loss of generality one can assume that $F$ is a finitely generated 
extension of $E^{(0\times1)}$. We choose an embedding of $F$ into $\c$. By 
very construction of the map $g^{(k)}:{_{K^{(0\times1)}}M}(G^{(0\times1)},
X^{(0\times1)})\rightarrow{_{K^{(k)}}M}(G^{(k)},X^{(k)})$ there are 
isomorphisms of Hodge structures with $\O_L$-operation:
$$
\begin{CD}
H_1(Y_\xi^{(0)}(\c),\z)^{\otimes_{\O_L}1-k}
\otimes_{\O_L}\bigwedge_{\O_L}^kH_1(Y_\xi^{(1)}(\c),\z)\\
@V{t^{(k)}}VV\\
H_1(Y_\xi^{(k)}(\c),\z)
\end{CD}
$$
Let us denote by $t_{dR}^{(k)}$ and $t_{et}^{(k)}$ the de Rham and 
\'etale realizations which are obtained from $t^{(k)}$ by composing with
the comparison isomorphisms:
$$
H_1^{dR}(Y_\xi^{(k)}/\c)\cong H_1(Y_\xi^{(k)}(\c),\c)
$$
and
$$
H_1(Y_\xi^{(k)}\times_FF^{ac},\zd)\cong H_1(Y_\xi^{(k)}(\c),\zd).
$$
The pairs $(t_{dR}^{(k)},t_{et}^{(k)})$ are Hodge cycles (on the abelian 
variety ${Y_\xi^{(0)}}^{\times_Fk-1}\times_FY_\xi^{(1)}\times_FY_\xi^{(k)}$)
in the sense of ~\cite[Paragraph 2]{deligne2}. By 
~\cite[Theorem 2.11]{deligne2} these are absolute Hodge cycles (on the 
aforementioned product of abelian varieties over $F^{ac}$). By 
~\cite[Proposition 2.5]{deligne2} this implies that the de Rham components 
$t_{dR}^{(k)}$ are horizontal with respect to the Gauss-Manin connection, in 
particular they descend to isomorphisms over $F^{ac}$:
$$
\begin{CD}
H_1^{dR}(Y_\xi^{(0)}/F^{ac})^{\otimes_{F^{ac}\otimes L}1-k}
\otimes_{F^{ac}\otimes L}
\bigwedge_{F^{ac}\otimes L}^kH_1^{dR}(Y_\xi^{(1)}/F^{ac})\\
@V{t_{dR}^{(k)}}VV\\
H_1^{dR}(Y_\xi^{(k)}/F^{ac}),
\end{CD}
$$ 
according to ~\cite[Corollary 2.7]{deligne2}. Moreover, the subgroup of 
$\Gal(F^{ac}/F)$ fixing $t_{dR}^{(k)}$ is the same as the subgroup of 
$\Gal(F^{ac}/F)$ fixing $t_{et}^{(k)}$, essentially because Galois conjugates 
of absolute Hodge cycles are again absolute Hodge cycles, 
~\cite[Proposition 2.9]{deligne2}. Lemma ~\ref{vvier} finishes our proof as 
$t_{et}^{(k)}$, is indeed rational over $F$, so that our horizontal map 
$t_{dR}^{(k)}$ descends to $F$:
$$
\begin{CD}
H_1^{dR}(Y_\xi^{(0)}/F)^{\otimes_{F\otimes L}1-k}
\otimes_{F\otimes L}\bigwedge_{F\otimes L}^kH_1^{dR}(Y_\xi^{(1)}/F)\\
@V{t_{dR}^{(k)}}VV\\
H_1^{dR}(Y_\xi^{(k)}/F)
\end{CD}
$$
too.
\end{proof}

\begin{rem}
As we can identify $G^{(0)}$ with a subgroup of $G^{(1)}$, we can identify 
$G^{(0\times1)}$ with the product $G^{(0)}\times G$ where $G$ is the kernel 
of the map $g^{(0)}:G^{(0\times1)}\rightarrow G^{(0)}$. We even get a product 
of Shimura data 
$(G^{(0\times1)},X^{(0\times1)})\cong(G^{(0)},X^{(0)})\times(G,X)$, where $X$ 
is the projection of $X^{(0\times1)}$ onto the $G$-factor.\\ 
Notice however, that $(G,X)$ is not the type of Shimura datum, that can arise 
from any $PEL$-datum, nor is it of Hodge type. This is because its weight 
homomorphism is trivial. In fact one can think of $M_\c(G,X)$ as a moduli 
space for the Hodge structures 
${V^{(0)}}^{\otimes_{\O_L}-1}\otimes_{\O_L}V^{(1)}$ which have weight zero.
\end{rem}

\section{Deformations of ordinary points}
\label{deform}
\subsection{Canonical Coordinates}
\label{caco}
In this subsection $L$ is a finite extension of $\q_p$, let $\O_L$ be its ring
of integers, $\gp$ its maximal ideal, and $\f=\O_L/\gp$ its residue field. We 
need some background material on Barsotti-Tate groups with $\O_L$-operation. 
Let $R$ be a noetherian local $\O_L$-algebra, complete with respect to the 
maximal ideal $\gm$. Assume that the residue field $k=R/\gm$ is an 
algebraically closed field extension of $\f$. Let $\G$ be a Barsotti-Tate 
group over $R$. Let $\iota:\O_L\rightarrow\End(\G)$ be a homomorphism. The 
pair $(\G,\iota)$ is called a $\O_L$-Barsotti-Tate group, if $\O_L$ acts on 
$\Lie \G$ by means of the structural morphism $\O_L\rightarrow R$, see 
~\cite{messing}, for the definition of $\Lie\G$. A few numerical invariants 
are noteworthy: the $\O_L$-height $\ho_{\O_L}\G$ of $\G$ is defined by 
$|\G[\gp]|=|\O_L/\gp|^{\ho_{\O_L}\G}$, and the dimension of $\G$ is the 
rank of the projective module $\Lie\G$. Every $\O_L$-Barsotti-Tate group sits 
in a unique exact sequence:
$$
\begin{CD}
0@>>>\G^\circ@>>>\G@>>>\G^{et}@>>>0
\end{CD}
$$
where $\G^\circ$ is connected and $\G^{et}$ is \'etale. One calls $\G$ 
ordinary if and only if $\dm\G={\ho}_{\O_L}\G^\circ$, in general one has an 
inequality $\dm\G\leq{\ho}_{\O_L}\G^\circ$. It follows from the 
Dieudonn\'e-Manin classification that over an algebraically closed field there 
is one and only one ordinary $\O_L$-Barsotti-Tate group of say dimension $a$ 
and $\O_L$-height $a+b$, cf. ~\cite[(29.8)]{hazewinkel}. Also, by the 
rigidity of \'etale covers there is one and only one \'etale 
$\O_L$-Barsotti-Tate group of given $\O_L$-height over any $R$. In fact the 
same is true at the other extreme:

\begin{lem}
\label{starr}
Let $R$ be as above and let $a$ be an integer. The category of 
$\O_L$-Barsotti-Tate groups $\G$ over $R$ with $a=\dm\G=\ho_{\O_L}\G$, is 
equivalent to the category of $\O_L$-Barsotti-Tate groups $\G$ over $k=R/\gm$ 
with $a=\dm\G=\ho_{\O_L}\G$.
\end{lem}
\begin{proof}
Let $R$ be artinian, let $I\subset R$ be an ideal of square zero, let us 
endow it with the trivial divided power structure. Consider a  
$\O_L$-Barsotti-Tate group $\G_0$ over $R_0=R/I$ with associated crystal
$D(\G_0)$, see ~\cite{messing}. Notice that the value $D(\G_0)_R$ over $R$ is 
a free $R\otimes\O_L$-module of rank $a$. We have to show that $\G_0$ has a 
unique lift to a $\O_L$-Barsotti-Tate group $\G$ over $R$. So consider all 
sequences:
$$
\begin{CD}
0@>>>Fil^1@>>>D(\G_0)_R@>>>\Lie\G@>>>0
\end{CD}
$$
where the quotient $\Lie\G$ is a free $R$-module of rank $a$ on which
the $\O_L$ action factors through $\O_L\rightarrow R$. This last condition 
actually means that $\Lie\G$ is a quotient of the $R$-module 
$R\otimes_{\O_L}\O_L^a=R^a$. Due to rank reasons we then have equality.
\end{proof}

From this it follows easily that we have canonical coordinates for
ordinary $\O_L$-Barsotti-Tate groups, we write $\Sigma/\O_L$ 
for the Lubin-Tate one, $\Sd/\O_L$ for its formal group, by $\Sd(R)$ we
mean the set $\gm$ which is given a $\O_L$-module structure by the group 
law of $\Sd$.

\begin{lem}
\label{coordinate}
Let $k=R/\gm$ be as above and let $\G_0$ be an ordinary $\O_L$-Barsotti-Tate 
group over $k/\f$. Consider the $\O_L$-modules:
$$
T'=\Hom_{\O_L}(L/\O_L,\G_0)
$$
and
$$
T''=\Hom_{\O_L}(\Sigma\times_{\O_L}k,\G_0).
$$
There exists an equivalence of categories between lifts $\G/R$ of $\G_0$ and
maps $\phi\in\Hom_{\O_L}(T',T'')\otimes_{\O_L}\Sd(R)$, the equivalence 
being established by decreeing $\G$ to be the following push out:
$$\begin{CD}
0@>>>\G^\circ@>>>\G@>>>\G^{et}@>>>0\\
@AAA@A{\phi}AA@AAA@A=AA@AAA\\
0@>>>T'@>>>T'\otimes_{\O_L}L
@>>>T'\otimes_{\O_L}L/\O_L@>>>0\\
\end{CD},$$
where $\G^\circ=T''\otimes_{\O_L}\Sigma\times_{\O_L}R$.
\end{lem}
\begin{proof}
The proof is word for word the same as in ~\cite[Paragraph 2]{katz}, so we 
are brief. It is enough to consider lifts $\G$ over artinian $\O_L$-algebras 
$R$. According to lemma ~\ref{starr} the group $\G^\circ$ is canonically 
isomorphic to $T''\otimes_{\O_L}\Sigma\times_{\O_L}R$. Moreover, by the 
rigidity of quasi-isogenies the extension class of:
$$\begin{CD}
0@>>>T''\otimes_{\O_L}\Sigma\times_{\O_L}R@>>>
\G@>>>T'\otimes_{\O_L}L/\O_L@>>>0
\end{CD}$$
is torsion, say killed by $p^n$. It follows that every element in
$$
\Ext_{\O_L}^1(T'\otimes_{\O_L}L/\O_L,
T''\otimes_{\O_L}\Sigma\times_{\O_L}R)
$$
is induced from an element in
$$
\Hom_{\O_L}(T'\otimes_{\O_L}p^{-n}\O_L/\O_L,
T''\otimes_{\O_L}\Sigma\times_{\O_L}R)
$$
which is unique as
$$
\Hom_{\O_L}(T'\otimes_{\O_L}L/\O_L,
T''\otimes_{\O_L}\Sigma\times_{\O_L}R)=0.
$$
\end{proof}

In the sequel we will frequently use that the category of Barsotti-Tate 
groups over $R$ is equivalent to the category of projective systems of 
Barsotti-Tate groups over $R/M$ where $M$ runs through all $\gm$-primary 
ideals ~\cite[Chapter II, Lemma(4.16)]{messing}, in particular we can talk 
about the generic fibre of a deformation. In the ordinary case these 
can be obtained as follows:

\begin{lem}
\label{trivial}
Let $\G_0$ be an ordinary $\O_L$-Barsotti-Tate group over $k/\f$, with 
dimension $a$ and $\O_L$-height $a+b$. Let $T'$, $T''$, $R/\O_L$, and 
$\phi\in\Hom_{\O_L}(T',T'')\otimes_{\O_L}\Sd(R)$ with corresponding 
deformation $\G/R$ be as in lemma ~\ref{coordinate}. Let $K/L$ be a field 
extension together with a $\O_L$-linear homomorphism $R\rightarrow K$. Pick 
$\O_L$-bases $e_1',\dots,e_b'$ of $T'$, and $e_1'',\dots,e_a''$ of $T''$, and 
let $\phi_{i,j}\in\Sd(R)$ be the entries of the matrix which represents 
$\phi$. Let further $\G_{i,j}$ be the $\O_L$-Barsotti-Tate group over $R$ 
defined by the push out:
\begin{equation}
\label{ijott}
\begin{CD}
0@>>>\Sigma\times_{\O_L}R@>>>\G_{i,j}@>>>L/\O_L@>>>0\\
@AAA@A{\phi_{i,j}}AA@AAA@A=AA@AAA\\
0@>>>{\O_L}@>>>L@>>>L/\O_L@>>>0\\
\end{CD},
\end{equation}
and let $c_1:\Gal(L^{ac}/L)\rightarrow\O_L^\times$ be the character obtained 
from the Tate module $\O_L(1)$ of $\Sigma\times_{\O_L}L^{ac}$. Then the 
operation of $\Gal(K^{ac}/K)$ on the Tate module of $\G\times_RK^{ac}$ is:
$$
\left(\begin{matrix}
c_1E_a&C\\
0&E_b
\end{matrix}\right)
$$
where $C=(c_{i,j})$ is the $a\times b$-matrix such that 
$c_{i,j}:\Gal(K^{ac}/K)\rightarrow\O_L(1)$ is any $1$-cocycle representing the 
continuous cohomology class in $H_{cont}^1(\Gal(K^{ac}/K,\O_L(1))$ which is 
obtained from the extension ~\eqref{ijott}.
\end{lem}
\begin{proof}
By addition of extension classes we may assume without 
loss of generality that at most one entry $\phi_{i_0,j_0}$ 
is non-zero. In this case $\G$ is isomorphic to 
$\G_{i_0,j_0}\oplus(\Sigma\times_{\O_L}R)^{a-1}\oplus(L/\O_L)^{b-1}$ 
and the assertion is obvious. 
\end{proof}

\subsection{Canonical Lifts}
\label{lift}
We turn to consequences for abelian schemes with $\O_L$-operation, where $L$
is a $CM$-field. In analogy to subsection ~\ref{caco} we fix a rational prime 
$p$ and assume that, every prime of $L^+$ over $p$ is split in $L$, hence 
there exists a set $\pi=\{\gq_1,\dots,\gq_m\}$ of primes of $L$ over $p$, 
such that $\{\gq_1,\dots,\gq_m,\gq_1^*,\dots,\gq_m^*\}$ exhaust all the 
primes of $L$ over $p$, consequently we have a decomposition:
\begin{equation}
\label{Kompost}
\O_L\otimes\z_p\cong\bigoplus_{i=1}^m\O_{L_{\gq_i}}
\oplus\bigoplus_{i=1}^m\O_{L_{\gq_i^*}}.
\end{equation}
We write $e_1,\dots,e_m,e_1^*,\dots,e_m^*\in\O_L\otimes\z_p$ for the 
corresponding idempotents. Moreover for $i\in\{1,\dots,r\}$, we want to fix 
embeddings $\sigma_i:L_{\gq_i}\rightarrow\q_p^{ac}$, where we fix an 
algebraic closure $\q_p^{ac}$ of $\q_p$. We put $\O_p^{ac}$ for the integers 
in $\q_p^{ac}$, and $\gP|p$ for the maximal ideal. Finally, $E$ is the field 
generated by the $\sigma_i(L_{\gq_i})$'s, with ring of integers $\O_E$, 
maximal ideal $\gp$, and residue field $\f_\gp$.\\
Let $S$ be a base scheme over $\O_E$ and let $A$ be an abelian scheme
over $S$, and let $\lambda$ be a $p$-principal quasipolarization on $A/S$. We 
will say that some operation $\iota:\O_L\rightarrow\End(A)$, makes 
$(A,\lambda)$ into a $\O_L$-abelian scheme if and only if for every 
$i\in\{1,\dots,r\}$ the induced operation on the projective $\O_S$-module 
$e_i\Lie A$ coincides with scalar multiplication by means of the map 
$\sigma_i:\O_{L_{\gq_i}}\rightarrow\O_E\rightarrow\Gamma(S,\O)$, and if 
$e_i\Lie A=0$ for the remaining $i\in\{r+1,\dots,m\}$. More specifically let 
$k$ be an algebraically closed field extension of $\f_\gp$. Then a 
$\O_L$-abelian scheme $(A,\lambda,\iota)$ over $k$ gives rise to 
$\O_{L_{\gq_i}}$-Barsotti-Tate groups $A[\gq_i^\infty]$. We will say that 
$(A,\lambda,\iota)$ is ordinary if this holds for all of the 
$A[\gq_i^\infty]$. In this case we can apply the Serre-Tate canonical 
coordinates to study deformations of the $\O_L$-abelian scheme 
$(A,\lambda,\iota)$ over any $\Spec k\hookrightarrow\Spec R$, such that $R$ 
is a noetherian local $\O_E$-algebra that is complete with respect to its 
maximal ideal. In particular let us look at the unique unramified extension 
$B/E$ of complete discretely valued fields that has $k/\f_\gp$ as residue 
field extension. The lift $(\tilde A,\lambda,\iota)$ to $\O_B$ with the 
canonical coordinates of all the $\tilde A[\gq_i^\infty]$ the trivial ones is 
called the canonical lift. We have the following important fact:

\begin{lem}
\label{rig}
Fix $\pi=\{\gq_1,\dots,\gq_m\}$ and $\sigma_i:L_{\gq_i}\rightarrow\q_p^{ac}$ 
for $i\in\{1,\dots,r\}$ as above. Let $(A,\lambda,\iota)$ be a $\O_L$-abelian 
variety over the algebraically closed field extension $k$ over $\f_\gp$. Let 
$(\tilde A,\lambda,\iota)$ be the canonical lift over $\O_B$. Then 
$\End_{\O_L}(\tilde A)=\End_{\O_L}(A)$, as $\z$-algebras with involution.
\end{lem}
\begin{proof}
The proof in ~\cite[Chapter V,Theorem(3.3)]{messing} translates word 
for word to our situation.
\end{proof}

\section{The integral model $\M^{(0\times1)}$}
\label{good}
\subsection{The moduli problem}
\label{moduli}
Recall the two $CM$ traces $\Phi^{(0)}$ and $\Phi^{(n)}$ for $L$, as in 
section ~\ref{Hodge}, and recall also our choice of $(V^{(1)}$, $V^{(0)}$, 
$\psi^{(1)}$, $\psi^{(0)}$, $h^{(1)}$, $h^{(0)})$ giving rise to Shimura data 
$(G^{(0\times1)},X^{(0\times1)})$ and of the level $l\geq3$. We fix a prime 
$\gP|p$ in $\q^{ac}$, the algebraic closure of $\q$ in $\c$, and 
assume ($\spadesuit$) the following:
\begin{itemize}
\item
the pairings $\psi^{(1)}$ and $\psi^{(0)}$ induce $\z_{(p)}$-valued 
perfect pairings on $V_{\z_{(p)}}^{(1)}$ and $V_{\z_{(p)}}^{(0)}$,
\item
$p$ is coprime to $l$,
\item
there exists a set $\pi$ of primes of $L$ over $p$ such that 
$$|\Phi^{(0)}|=\{\sigma:L\rightarrow\q^{ac}|\sigma^{-1}(\gP)\in\pi^*\},$$
\item
if $|\Phi^{(n)}|-|\Phi^{(0)}|=\{\sigma_1,\dots,\sigma_r\}$ then 
$\gq_1=\sigma_1^{-1}(\gP),\dots,\gq_r=\sigma_r^{-1}(\gP)$ are pairwise 
distinct prime ideals of $L$.
\end{itemize}
We write $\{\gq_i|i=r+1,\dots,m\}$ for the remaining primes in $\pi$. The 
primes in $\pi^*$ are nothing else then $\{\gq_1^*,\dots,\gq_m^*\}$, 
consequently we have a decomposition as in ~\eqref{Kompost} with corresponding
idempotents $e_1$, $\dots$, $e_m$, $e_1^*$, $\dots$, 
$e_m^*\in\O_L\otimes\z_p$. We write $E$ for the field generated by the 
$\sigma_i(L)$'s, we write $\gp$ for the prime ideal induced in $E$ by $\gP$, 
and denote as usual by $\O_{E_\gp}$ the integers in the completion of $E$ at 
$\gp$, notice that $E$ contains the reflex field $E^{(0\times1)}$ and that 
$E_\gp$ may well be ramified over $E_\gp^{(0\times1)}$. The moduli 
interpretation for $M^{(0\times1)}$ that we will give is only defined over 
the extension $\O_{E_\gp}$.\\
The first two $\spadesuit$-conditions imply that the group $K^{(0\times1)}$, 
as introduced in ~\eqref{co}, allows a factorization 
${K^{(0\times1)}}^pK_p^{(0\times1)}$, where 
${K^{(0\times1)}}^p\subset G^{(0\times1)}(\ap)$ is compact open, and
$K_p^{(0\times1)}=G_\z^{(0\times1)}(\z_p)$ looks like $\z_p^\times\times
\prod_{i=1}^m\O_{L_{\gq_i}}^\times\times\prod_{i=1}^m\GL(n,\O_{L_{\gq_i}})$. 
We begin by introducing a moduli space of abelian varieties. We need the 
following set valued functor, over a $\O_{E_\gp}$-scheme $S$ its points 
consist of:

\begin{itemize}
\item[(a')]
$Y^{(1)}$ and $Y^{(0)}$, abelian schemes over $S$, up to $\z_{(p)}$-isogeny,
\item[(b')]
operations $\iota^{(1)}:\O_L\rightarrow\End(Y^{(1)})\otimes\z_{(p)}$
and $\iota^{(0)}:\O_L\rightarrow\End(Y^{(0)})\otimes\z_{(p)}$, such that
for $k\in\{0,1\}$ the $\O_S$-modules $e_i\Lie Y^{(k)}$ are projective of 
rank $1$ if $k=1$ and $i\in\{1,\dots,r\}$ and of rank $0$ otherwise, 
moreover in the former case the $\O_L$-operation on $e_i\Lie Y^{(1)}$ is 
given by $\sigma_i$.
\item[(c')]
$\lambda^{(0\times1)}$, a homogeneous class of polarizations on 
$Y^{(0)}\times Y^{(1)}$ containing a representative of degree in 
$\z_{(p)}^\times$,
\item[(d')]
level-${K^{(0\times1)}}^p$-structure $\ebar^{(0\times1)}$, i.e. for some 
choice of geometric point $\xi$ of $S$ one has a $\pi_1(S,\xi)$-invariant  
${K^{(1)}}^p$- (resp. ${K^{(0)}}^p$-) class of $\O_L\otimes\ap$-linear 
symplectic similitudes:
$$
\eta^{(1)}:V^{(1)}\otimes\ap\rightarrow H_1(Y_\xi^{(1)},\ap)
$$
(resp.
$$
\eta^{(0)}:V^{(0)}\otimes\ap\rightarrow H_1(Y_\xi^{(0)},\ap)
$$
), of which the multipliers agree. 
\end{itemize}
The above functor is representable by a $\O_{E_\gp}$-scheme $\M^{(0\times1)}$.
The general fibre of it is canonically isomorphic to
\begin{equation}
\label{haesslich}
M^{(0\times1)}=\coprod_i{_{K^{(0\times1)}}M}
(G_i^{(0\times1)},X^{(0\times1)})\times_{E^{(0\times1)}}E_\gp
\end{equation}
where $i$ indexes all the locally trivial $G^{(0\times1)}$-torsors, and where
$G_i^{(0\times1)}$ is the automorphism group of the $i$th 
$G^{(0\times1)}$-torsor. It is clear that the constructions in section 
~\ref{Hodge}, can be applied to each of the Shimura varieties in
~\eqref{haesslich} at a time, i.e. there are Shimura data 
$(G_i^{(k)},X^{(k)})$ corresponding to each of the locally trivial 
$G^{(0\times1)}$-torsors and there are maps
\begin{equation}
\label{velfi}
g^{(k)}:{_{K^{(0\times1)}}M}(G_i^{(0\times1)},X^{(0\times1)}) 
\rightarrow{_{K^{(k)}}M}(G_i^{(k)},X^{(k)})
\end{equation}
generalizing the morphism ~\eqref{velf}. In particular, by our conventions on 
the compact open subgroups $K^{(0\times1)}$ and $K^{(k)}$, we obtain 
homogeneously polarized abelian schemes with properties as expressed in 
lemma ~\ref{vvier} over each of the Shimura varieties 
${_{K^{(0\times1)}}M}(G_i^{(0\times1)},X^{(0\times1)})$. We will continue to 
denote these by $Y^{(k)}$. The same remark applies to the induced classifying 
maps ~\eqref{fuenf} which generalize to give maps:
\begin{equation}
\label{fuenfi}
g^{(k)}:M^{(0\times1)}\rightarrow\A_{g_k,d_k,l}.
\end{equation}
Here is $g_k=[L^+:\q]\binom{n}{k}$, and $d_k$ is the degree of some choice
of an effective polarization within the homogeneous class $\lambda^{(k)}$.

\begin{lem}
If $\spadesuit$ holds, then $\M^{(0\times1)}$ is smooth over $\O_{E_\gp}$.
\end{lem}
\begin{proof}
We use the concept of a local model $M^{loc}$ over $\O_{E_\gp}$, 
following the method of ~\cite{rapoport}. In the case at hand, the functor 
which $M^{loc}$ represents, can be described as follows, over 
$S/\Spec\O_{E_\gp}$ the points consist of pairs $(t,\phi)$ where $t$ is a 
$(\bigoplus_{i=1}^m\O_{L_{\gq_i}})\otimes_{\z_p}\O_S$ module, and $\phi:
 (\bigoplus_{i=1}^m\O_{L_{\gq_i}})^n\otimes_{\z_p}\O_S\rightarrow t$ is a 
$(\bigoplus_{i=1}^m\O_{L_{\gq_i}})\otimes_{\z_p}\O_S$-linear surjective map 
such that:
\begin{itemize}
\item
$e_it$ is a line bundle on $S$, for all $i\in\{1,\dots,r\}$, moreover 
$\O_{L_{\gq_i}}$ acts on it by $\sigma_i$.
\item
$e_it=0$  for all $i\in\{r+1,\dots,m\}$
\end{itemize}
One sees that $M^{loc}$ is smooth, as the fibres are isomorphic to
$(\p^{n-1})^{\times r}$.
\end{proof}

\subsection{Stratification}
Let $k$ be an algebraically closed field over $\f_\gp$ and let 
$\xi:\Spec k\rightarrow\M^{(0\times1)}$ be a point corresponding to data 
$(Y_\xi^{(1)},Y_\xi^{(0)},\dots)$. If $i\in\{r+1,\dots,m\}$, then the 
$\O_{L_{\gq_i}}$-Barsotti-Tate group $Y_\xi^{(1)}[\gq_i^\infty]$ is \'etale, 
if $i\in\{1,\dots,r\}$ it is one-dimensional, hence isomorphic to 
$$
G^{\O_{L_{\gq_i}}}_{1,n-f_i-1}\oplus(L_{\gq_i}/\O_{L_{\gq_i}})^{f_i}
$$
for some $f_i\in\{0,\dots,n-1\}$, here the notation is from 
~\cite[(29.8)]{hazewinkel}. The formal $\O_{L_{\gq_i}}$-module is in this 
case isomorphic to
$$
G^{\O_{L_{\gq_i}}}_{1,n-f_i-1}.
$$
We call $\xi$ ordinary if $f_1=\dots=f_r=n-1$, i.e. if and only if all the 
$\O_{L_{\gq_i}}$-Barsotti-Tate groups $Y_\xi^{(1)}[\gq_i^\infty]$ are 
ordinary in the sense of section ~\ref{deform}. The ordinary locus is
Zariski open by ~\cite{richartz}. We let $\M_{ord}^{(0\times1)}$ be the open 
subscheme of $\M^{(0\times1)}$ obtained by removing the non-ordinary locus in 
the special fibre.

\section{An extension theorem}
\label{lego}
We continue the study of the abelian schemes $Y^{(k)}$, we begin with a lemma 
on the degree of $Y^{(k)}$:

\begin{lem}
\label{prinz}
Let $(V^{(1)},V^{(0)},\psi^{(1)},\psi^{(0)})$ and corresponding 
$(V^{(k)},\psi^{(k)})$, be as in section ~\ref{Hodge}.
If $\psi^{(1)}$ and $\psi^{(0)}$ induce $\z_{(p)}$-valued 
perfect pairings on $V_{\z_{(p)}}^{(1)}$ and $V_{\z_{(p)}}^{(0)}$, then
$\psi^{(k)}$ induces a $\z_{(p)}$-valued perfect pairing on 
$V_{\z_{(p)}}^{(k)}$.
\end{lem}
\begin{proof}
The data $\psi^{(k)}$ gives rise to a $\z_{(p)}$-valued perfect pairing on 
the $\O_L\otimes\z_{(p)}$-module $V_{\z_{(p)}}^{(k)}$ if and only if
$$
\Psi^{(k)}:V_{\z_{(p)}}^{(k)}\times V_{\z_{(p)}}^{(k)}
\rightarrow\D_L^{-1}\otimes\z_{(p)}
$$
has this property, where $\D_L$ is the different of $\O_L$. If one has this
for $k\in\{0,1\}$, then the same follows for the exterior pairing 
$$
\bigwedge_{\O_L}^kV_{\z_{(p)}}^{(1)}\times
\bigwedge_{\O_L}^kV_{\z_{(p)}}^{(1)}\rightarrow
\D_L^{-k}\otimes\z_{(p)},
$$
induced by $\Psi^{(1)}$, and for the $1-k$-fold self-product of $\Psi^{(0)}$:
$$
{V_{\z_{(p)}}^{(0)}}^{\otimes_{\O_L}1-k}
\times{V_{\z_{(p)}}^{(0)}}^{\otimes_{\O_L}1-k}
\rightarrow\D_L^{k-1}\otimes\z_{(p)}.
$$
By taking the product once more we obtain that 
$$
\Psi^{(k)}:V_{\z_{(p)}}^{(k)}\times V_{\z_{(p)}}^{(k)}
\rightarrow\D_L^{-1}\otimes\z_{(p)}
$$
is perfect, which is what we wanted.
\end{proof}

\subsection{Extension of $Y^{(k)}[\gq_i^\infty]$}
\label{formal}
With these gadgets we are now in a position to give a Serre-Tate analog 
of the tensor constructions of section ~\ref{Hodge}.
Let us start with an algebraically closed extension $k/\f_\gp$ and an ordinary
$k$-valued point $\xi:\Spec k\rightarrow\M_{ord}^{(0\times1)}$, that is 
represented by the tuple $(Y_\xi^{(1)}$, $Y_\xi^{(0)}$, $\iota^{(1)}$, 
$\iota^{(0)}$, $\lambda^{(0\times1)}$, $\dots)$. Our aim in this subsection is 
to define a certain Barsotti-Tate group: 
$$
\G^{(k)}=\bigoplus_{i=1}^m\G^{(k)}[\gq_i^\infty]
\oplus\bigoplus_{i=1}^m\G^{(k)}[\gq_i^{*\infty}]
$$
on the universal deformation space $\Df_\xi=\Spec R$ of $\xi$, in such a 
manner that the generic fibre of it matches our $Y^{(k)}[\gq_i^\infty]$. 
Here we use the usual algebraization results of polarized formal abelian 
schemes to freely switch between $\Spec R$ and $\Spf R$. Note that according 
to lemma ~\ref{coordinate} we have a canonical isomorphism 
$$
\Spf R\cong\prod_{i=1}^r\Hom_{\O_{L_{\gq_i}}}({T_i^{(1)}}',{T_i^{(1)}}'')
\otimes_{\O_{L_{\gq_i}}}\Sd_i\times_{\O_{L_{\gq_i}},\sigma_i}\O_B,
$$
where $\O_B\subset B$ is again the ring of integers in the unique complete
unramified extension $B$ of $E_\gp$ inducing the residue extension 
$k/\f_{\gp}$. It goes without saying that $\G^{(k)}[\gq_i^{*\infty}]$ will 
be the Serre dual of $\G^{(k)}[\gq_i^\infty]$. It is then meaningful to put:
$$
{T_i^{(k)}}'=\Hom_{\O_{L_{\gq_i}}}
(L_{\gq_i}/\O_{L_{\gq_i}},Y^{(k)}[\gq_i^\infty])
$$
and
$$
{T_i^{(1)}}''=\Hom_{\O_{L_{\gq_i}}}(\Sigma_i
\times_{\O_{L_{\gq_i}},\sigma_i}k,Y^{(1)}[\gq_i^\infty]),
$$
for $i\in\{1,\dots,m\}$, and $k\in\{0,1\}$, and further
$$
{T_i^{(k)}}'={{T_i^{(0)}}'}^{\otimes_{\O_{L_{\gq_i}}}1-k}
\otimes_{\O_{L_{\gq_i}}}\bigwedge_{\O_{L_{\gq_i}}}^k{T_i^{(1)}}'
$$
and
$$
{T_i^{(k)}}''={{T_i^{(0)}}'}^{\otimes_{\O_{L_{\gq_i}}}1-k}
\otimes_{\O_{L_{\gq_i}}}{T_i^{(1)}}''\otimes_{\O_{L_{\gq_i}}}
\bigwedge_{\O_{L_{\gq_i}}}^{k-1}{T_i^{(1)}}'
$$
Let us also introduce a map:
$$
g_{can,i}^{(k)}:\Hom_{\O_{L_{\gq_i}}}({T_i^{(1)}}',{T_i^{(1)}}'')
\rightarrow
\Hom_{\O_{L_{\gq_i}}}({T_i^{(k)}}',{T_i^{(k)}}'')
$$
which is defined by contraction of elements, sending 
$\phi^{(1)}:{T_i^{(1)}}'\rightarrow{T_i^{(1)}}''$ to the map 
$\phi^{(k)}:{T_i^{(k)}}'\rightarrow{T_i^{(k)}}''$ defined by
$$
x_0^{1-k}x_1\wedge\dots\wedge x_k\mapsto
\sum_{\nu=1}^k(-1)^{\nu-1}x_0^{1-k}\phi^{(1)}(x_\nu)
x_1\wedge\dots\wedge x_{\nu-1}\wedge x_{\nu+1}\wedge\dots\wedge x_k,
$$
where $x_0\in{T_i^{(0)}}'-\{0\}$ and $x_1,\dots,x_k\in{T_i^{(1)}}'$. Notice 
that ${T_i^{(k)}}''$ and $g_{can,i}^{(k)}$ are only non-zero if 
${T_i^{(1)}}''$ is non-zero, which happens for all $i$ in $\{1,\dots,r\}$. To 
define a $\O_{L_{\gq_i}}$-Barsotti-Tate group $\G^{(k)}[\gq_i^\infty]$ of 
$\O_{L_{\gq_i}}$-height $\binom{n}{k}$ over $\Df_\xi$ we consider
$$
{T_i^{(k)}}''\otimes_{\O_{L_{\gq_i}}}
\Sigma_i\times_{\O_{L_{\gq_i}},\sigma_i}k\oplus
{T_i^{(k)}}'\otimes_{\O_{L_{\gq_i}}}L_{\gq_i}/\O_{L_{\gq_i}}
$$
and deform it by $g_{can,i}^{(k)}$. We can finish this subsection with the
following two auxiliary lemmas:

\begin{lem}
\label{vdrei}
Let $\pi_1$ be a profinite group. Let $\rho$ be a continuous representation 
thereof into the group $\GL(n,A)$, where $A$ is a ring, separated and complete
with respect to the $M$-adic topology where $M$ is an ideal. Assume that 
$\rho$ has the form
$$
\left(\begin{matrix}
c_1&c_2&\dots&c_n\\
0&1&\dots&0\\
\vdots&\vdots&\ddots&\vdots\\
0&0&\dots&1
\end{matrix}\right)
$$
where $c_1:\pi_1\rightarrow A^\times$ is a continous character and where
the maps $c_2$, $\dots$, $c_n:\pi_1\rightarrow A$  are continous $1$-cocycles 
of $\pi_1$ with coefficients in $c_1$ (i.e. representatives of elements in 
$H_{cont}^1(\pi_1,c_1)$). Then for every $k\in\{0,\dots,n\}$ the 
representation $\bigwedge^k\rho:\pi_1\rightarrow\GL({\binom{n}{k}},A)$ has 
the form
$$
\left(\begin{matrix}
c_1E_{\binom{n-1}{k-1}}&C\\
0&E_{\binom{n-1}{k}}
\end{matrix}\right)
$$
where the $\binom{n-1}{k-1}\times\binom{n-1}{k}$-matrix $C$ can be 
described as follows: the rows are indexed by the set of $k-1$-element 
subsets $I=\{i_2,\dots,i_k\}\subset\{2,\dots,n\}$, the columns are 
indexed by the set of $k$-element subsets 
$J=\{i_1,\dots,i_k\}\subset\{2,\dots,n\}$, and the entry in the 
$I$'th row and $J$'th column is equal to
$$
\begin{cases}
(-1)^{\nu-1}c_{i_\nu}&I=J-\{i_\nu\}\\
0&\text{otherwise}
\end{cases}
$$
\end{lem}
\begin{proof}
Denote the standard basis of $A^n$ by $e_1,\dots,e_n$ the subspace
spanned by $e_1\wedge e_{i_2}\dots\wedge e_{i_k}$ is $\bigwedge^k\rho$ 
invariant and the quotient carries the trivial action. This show that the 
diagonal blocks are as asserted. To check the cocycle matrix note that:
\begin{eqnarray*}
&&\bigwedge^k\rho(e_{i_1}\wedge\dots\wedge e_{i_k})
-e_{i_1}\wedge\dots\wedge e_{i_k}\\
&&=\sum_{\nu=1}^k(-1)^{\nu-1}c_{i_\nu}e_1\wedge e_{i_2}\dots\wedge 
e_{i_{\nu-1}}\wedge e_{i_{\nu+1}}\wedge\dots\wedge e_{i_k}
\end{eqnarray*}
\end{proof}
\begin{lem}
\label{vzehn}
Let the tuple $(Y_\xi^{(1)}$, $Y_\xi^{(0)}$, $\iota^{(1)}$, $\iota^{(0)}$, 
$\lambda^{(0\times1)}$, $\dots)$ over $k/\f_\gp$ be as before, and let
$\Df_\xi\cong\Spec R$ be the universal deformation space.
Let $\G^{(0)}$, $\G^{(1)}$, and $\G^{(k)}$ be the Barsotti-Tate groups over 
$R$ constructed above. Then for every $i\in\{1,\dots,m\}$ the (\'etale!) 
Barsotti-Tate group with $\O_{L_{\gq_i}}$-operation: 
$$
{(\G^{(0)}[\gq_i^\infty]\times_RR[\frac{1}{p}])}^{\otimes_{\O_{L_{\gq_i}}}1-k}
\otimes_{\O_{L_{\gq_i}}}\bigwedge_{\O_{L_{\gq_i}}}^k
(\G^{(1)}[\gq_i^\infty]\times_RR[\frac{1}{p}])
$$
is canonically isomorphic to $\G^{(k)}[\gq_i^\infty]\times_RR[\frac{1}{p}]$.
\end{lem}
\begin{proof}
We may regard these objects over $R[\frac{1}{p}]$ as representations of
$\pi_1(\Spec R[\frac{1}{p}],\Spec k)$, now use lemma 
~\ref{vdrei}, and lemma ~\ref{trivial}.
\end{proof}

\subsection{Extension of $Y^{(k)}$}
\label{algebraic}

We continue to assume that $\spadesuit$ is valid:

\begin{thm}
The abelian scheme $Y^{(k)}$ over $M^{(0\times1)}$ extends to an abelian 
scheme $\Y^{(k)}$ over the whole of $\M_{ord}^{(0\times1)}$. It inherits a 
$\O_L$-operation from the $\O_L$-operation $\iota^{(k)}$ on $Y^{(k)}$ in a 
unique way.
\end{thm}
\begin{proof}
As $\M_{ord}^{(0\times1)}$ is disconnected we consider the connected 
components separetely, let $\N$ be one of them, it is an integral scheme, 
write $N$ for the generic fibre. We consider the morphism 
$g^{(k)}:N\rightarrow\A_{g_k,d_k,l}$, obtained from the abelian scheme 
$Y^{(k)}$ together with a choice of effective polarization in the homogeneous
class $\lambda^{(k)}$, as we did in ~\eqref{fuenf}. Here note that we can 
choose the effective polarization in the class $\lambda^{(k)}$ to have a 
degree $d_k$ coprime to $p$, because of lemma ~\ref{prinz}. Let $\N_0$ be the 
normalization of the schematic closure of the graph in 
$\N\times\A_{g_k,d_k,l}$, let $\chi$ be the projection from $\N_0$ to the 
$\N$-component. Let $\Nbar_0$, $\Nbar$, and $\Abar_{g_k,d_k,l}$, be the 
special fibres of $\N_0$, $\N$, and $\A_{g_k,d_k,l}$.\\

We want to prove that the fibres of $\chi$ are all $0$-dimensional. By the 
semicontinuity of fibre dimensions it is enough to consider 
$\f_\gp^{ac}$-valued points $x_0$ of $\Nbar_0$, lying over some
$\f_\gp^{ac}$-valued point $(x,y)$ of $\Nbar\times\Abar_{g_k,d_k,l}$. Write 
$R$ for the local ring of $\O_{E_\gp^{nr}}\times_{\O_{E_\gp}}\N_0$ at the 
closed point $x_0$ and let $I\subset R$ be the stalk at $x_0$ of the ideal 
sheaf to the closed immersion $$\f_\gp^{ac}\times_{\Nbar}\Nbar_0\hookrightarrow
\O_{E_\gp^{nr}}\times_{\O_{E_\gp}}\N_0.$$ Write $\Rd$ and $\Id$ for their 
completions at the maximal ideal to $x_0$. Consider the commutative 
diagram:

$$\begin{CD}
\Spec R@>>>
\O_{E_\gp^{nr}}\times_{\O_{E_\gp}}\N_0@>>>
\O_{E_\gp^{nr}}\times_{\O_{E_\gp}}\N\times\A_{g_k,d_k,l}\\
@AAA@AAA@AAA\\
\Spec(R/I)@>>>
\f_\gp^{ac}\times_{\Nbar}\Nbar_0@>>>
\f_\gp^{ac}\times\Abar_{g_k,d_k,l}\\
\end{CD}$$

Over $\Spec\Rd$ we have the pull-back of the universal abelian schemes 
$\Y^{(0)}$, $\Y^{(1)}$, and  $\Y^{(k)}$. From the composition 
$\Spec\Rd\rightarrow\Df_x$, we also have $\O_{L_{\gq_i}}$-Barsotti-Tate 
groups $\G^{(k)}[\gq_i^{\infty}]$ according to subsection ~\ref{formal}.
By lemma ~\ref{vzehn} the generic fibre of $\G^{(k)}[\gq_i^{\infty}]$ 
agrees canonically with the generic fibre of $Y^{(k)}[\gq_i^{\infty}]$, 
according to ~\cite[Theorem 4]{tate1} $\G^{(k)}[\gq_i^{\infty}]$ agrees with 
$\Y^{(k)}[\gq_i^{\infty}]$. It follows that $\Y^{(k)}[\gq_i^{\infty}]$ is
constant on $\Spec(\Rd/\Id)$ because $\G^{(k)}[\gq_i^{\infty}]$ is constant 
there. According to Serre-Tate, ~\cite[Chapter V,Theorem 2.3]{messing}
the polarized abelian scheme $(\Y^{(k)},\lambda^{(k)})$ is constant on 
$\Spec(\Rd/\Id)$, and on $\Spec(R/I)$ as well. This means that the natural 
map from the local ring $Q$ of $\f_\gp^{ac}\times\Abar_{g_k,d_k,l}$ at $y$ to 
$R/I$ factors through $y:Q\rightarrow\f_\gp^{ac}$. Now just notice that $R/I$ 
being the local ring of $\f_\gp^{ac}\times_{\Nbar}\Nbar_0$ at $x_0$ is finite 
over $Q$, so that $R/I$ is finite over $\f_\gp^{ac}$.\\

We next prove that $\chi$ is proper, it certainly suffices to show that the 
schematic closure of $N$ is proper over $\N$, we check this by using the 
valuative criterion of properness, ~\cite[Corollaire 7.3.10(ii)]{ega2}. Let 
$F$ be the function field of $N$. Let $x:\Spec F\rightarrow N$ be the generic 
point, and let $g^{(k)}\circ x=y:\Spec F\rightarrow M^{(0\times1)}$ be the 
composition. Let $R$ be a discrete valuation ring of $F$, dominating some 
local ring of $\O_\N$. We write $\gx:\Spec R\rightarrow\N$ for the 
corresponding morphism, and $\Y_{\gx}^{(0)}$ and $\Y_{\gx}^{(1)}$, (resp. 
$Y_x^{(0)}$ and $Y_x^{(1)}$), for the pull backs of the universal abelian 
schemes over $\M^{(0\times1)}$ via $\gx$ (resp. via $x$). Choose any prime 
$\ell$ different from $p$. The $\ell$-adic Tate modules of $Y_x^{(0)}$ and 
$Y_x^{(1)}$ are unramified. Therefore, if $Y_y^{(k)}$ denotes the pull back 
of $Y^{(k)}$ to $\Spec F$, the $\ell$-adic Tate module of $Y_y^{(k)}$ is 
unramified as well, by lemma ~\ref{vvier}. Due to N\'eron-Ogg-Safarevic's
criterion, ~\cite[Theorem 1]{serre3}, there exists an abelian scheme 
$\Y_{\gy}^{(k)}$ over $R$ extending $Y_y^{(k)}$, moreover $\Y_{\gy}^{(k)}$
inherits a polarization of degree $d_k$ and a level $l$-structure from 
$Y_y^{(k)}$. The corresponding $R$-valued point in the moduli space 
$\A_{g_k,d_k,l}$ establishes a $R$-valued point $\gy$ in the schematic 
closure of $N$, which lies over $\gx$.\\

We conclude that $\chi$ is an isomorphism, by using the Main Theorem of 
Zariski ~\cite[Corollaire 4.4.9]{ega3}.\\

The abelian schemes thus obtained have a $\O_L$-operation, for example 
because homomorphisms between abelian schemes over normal bases extend, 
~\cite[Chap.IX, Corollaire 1.4]{raynaud}.
\end{proof}

\begin{rem}
The $\Y^{(k)}$'s are $\O_L$-abelian schemes, because the structure of the 
$\O_L$-operation on the Lie algebra can be checked in the generic fibre.
\end{rem}

Over algebraically closed fields of characteristic $p$ one can clarify 
how the canonical lift of $Y^{(1)}$ relates to the canonical lift of 
$Y^{(k)}$:

\begin{lem}
\label{al}
Let $k$ be an algebraically closed field over $\f_\gp$. Let $(Y^{(1)}$, 
$Y^{(0)}$, $\iota^{(1)}$, $\iota^{(0)}$,$\lambda^{(0\times1)}$, 
$\ebar^{(0\times1)})$ be a $k$-valued point of $\M_{ord}^{(0\times1)}$. Let 
$(\tilde Y^{(1)}$, $\tilde Y^{(0)}$, $\dots)$ be the canonical lift over 
$\O_B$. Let $\gx:\Spec\O_B\rightarrow\M_{ord}^{(0\times1)}$ be the 
classifying morphism. Then the $\O_L$-abelian scheme $\Y^{(k)}_\gx$ over 
$\O_B$ is the canonical lift of its special fibre.
\end{lem}
\begin{proof}
We only have to check that all the $\O_{L_{\gq_i}}$-Barsotti-Tate groups 
split, again by ~\cite[Theorem 4]{tate1} this follows if the Galois 
representation splits. This is clear, by lemma ~\ref{vvier}.
\end{proof}

\section{Endomorphism ring of $Y^{(k)}$}
\label{ende}

In ~\cite{satake} the endomorphism ring of the generic $Y^{(k)}$ is studied. 
We extend this study to the special fibre of $\Y^{(k)}$. We start with
preliminary remarks on the Mumford-Tate group: Consider a $\c$-valued point
$(Y_\xi^{(1)}$, $Y_\xi^{(0)}$, $\iota^{(1)}$, $\iota^{(0)}$, 
$\lambda^{(0\times1)}$, $\dots)$ of 
${_{K^{(0\times1)}}M}(G^{(0\times1)},X^{(0\times1)})$, then
$$
V_\q^{(1)}=H_1(Y^{(1)}(\c),\q)
$$
is a Hodge structure with an operation of $L$ on it. Let us write 
$\MT\subset\GL(V_\q^{(1)}/L)$, for the smallest algebraic group over $L$
such that $\MT\times_{L,\sigma}\c$ contains for every embedding 
$\sigma:L\rightarrow\c$ the cocharacter 
$\mu_\sigma:\g_m\times\c\rightarrow\MT\times_{L,\sigma}\c$ given by 
$$
\mu_\sigma(z):x\mapsto
\begin{cases}
zx&x\in{V_\sigma^{(1)}}^{-1,0}\\
x&x\in{V_\sigma^{(1)}}^{0,-1}
\end{cases}
$$
the following is very well-known:
\begin{thm}
\label{rib}
If $\End_L^0(Y_\xi^{(1)})=L$ then the group $\MT$ to the Hodge structure with 
$L$-operation $V_\q^{(1)}=H_1(Y_\xi^{(1)},\q)$ is the full linear group
$\GL(V_\q^{(1)}/L)$, in particular $\End_L^0(Y_\xi^{(k)})=L$ for all $k$.
\end{thm}
\begin{proof}
We follow the ideas of ~\cite[Theorem 3]{ribet}. Let us remark that 
$V^{(1)}$ is semisimple as a representation of $\MT$. This is because any 
$\MT$-invariant subspace $W$ is a Hodge structure with $L$-operation. 
Then take the orthogonal complement $W^*$ with respect to $\psi^{(1)}$. This 
is again a Hodge structure with $L$-operation, and hence a $\MT$ 
subrepresentation complementary to $W$. So $\MT$ is reductive, as it has a 
faithful semisimple representation $\rho$, namely the natural action on the 
$L$ vector space $V_\q^{(1)}$. A similar argument gives that $\End_\MT(\rho)$ 
coincides with $\End_L(V^{(1)})=L$.\\

Upon base change via some $\sigma:L\rightarrow \c$ we are in a position to 
apply ~\cite[Proposition 5]{serre2} to $\MT\times_{L,\sigma}\c$. Here note 
that we may choose $\sigma$ in the set 
$|\Phi^{(n)}|-|\Phi^{(0)}|\neq\emptyset$, so that the group 
$\mu_\sigma(\g_m)$ is contained in $\MT\times_{L,\sigma}\c$.\\

The assertion on $\End_L^0(Y_\xi^{(k)})$ just follows as the $k$th exterior
power is an absolutely irreducible representation of $\GL(V_\q^{(1)}/L)$.
\end{proof}

Using canonical lifts we easily get an analog in positive characteristic:

\begin{cor}
Assume that the $\spadesuit$-conditions of subsection ~\ref{moduli} hold. 
Consider a point $\xi:\Spec k\rightarrow\M_{ord}^{(0\times1)}$, where 
$k$ is an algebraically closed field over $\f_\gp$. If 
$\End_L^0(\Y_\xi^{(1)})=L$, then $\End_L^0(\Y_\xi^{(k)})=L$, for all $k$.
\end{cor}
\begin{proof}
Let $\gx:\Spec\O_B\rightarrow\M_{ord}^{(0\times1)}$ correspond to the 
canonical lift of $(\Y_\xi^{(1)}$, $\Y_\xi^{(0)}$, $\iota^{(1)}$, 
$\iota^{(0)}$, $\lambda^{(0\times1)}$, $\dots)$, where $\O_B$ is as in 
subsection ~\ref{lift}. Choose an embedding $\O_B\hookrightarrow\c$. We have 
$\End_L^0(\Y_\gx^{(1)}\times_{\O_B}\c)=L$, because this algebra must be 
contained in $\End_L^0(\Y_\xi^{(1)})=L$. By theorem ~\ref{rib} we infer 
$\End_L^0(\Y_\gx^{(k)}\times_{\O_B}\c)=L$, and therefore 
$\End_L^0(\Y_\gx^{(k)})=L$. Now we apply lemmas ~\ref{rig} and ~\ref{al}.
\end{proof}

\end{document}